\documentclass[12pt]{amsart}
\usepackage{amssymb}
\usepackage{amsthm}
\usepackage{mathrsfs}
\usepackage{amssymb}
\usepackage{amsfonts}
\usepackage[english]{babel}
\usepackage[T1]{fontenc}
\usepackage[latin1]{inputenc}
\usepackage{fullpage}
\usepackage{amsmath}
\usepackage[all]{xy}
\usepackage{stmaryrd}

\newtheorem{theorem}{Theorem}[section]
\newtheorem{lemma}[theorem]{Lemma}
\newtheorem{corollary}[theorem]{Corollary}
\theoremstyle{definition}
\newtheorem{definition}[theorem]{Definition}

\theoremstyle{remark}
\newtheorem{remark}[theorem]{Remark}

\numberwithin{equation}{section}

\begin{document}

\title{Brown-York mass and positive scalar curvature I \\ - First eigenvalue problem and its applications}

\author{Wei Yuan}
\address{Department of Mathematics, Sun Yat-sen University, Guangzhou, Guangdong 510275, China}
\email{gnr-x@163.com}



\thanks{This work was supported by NSFC (Grant No. 11601531, No. 11521101) and the Fundamental Research Funds for the Central Universities (Grant No. 2016-34000-31610258).}


\begin{abstract}
In this article, we investigate the connection between scalar curvature and first eigenfunctions via positive mass theorem for Brown-York mass. For compact manifolds with nice boundary, we show that a sharp inequality holds for first eigenfunctions when posing appropriate assumptions on scalar curvature and first eigenvalue. This inequality implies that for a compact $n$-dimensional manifold with boundary, its first eigenvalue is no less that $n$, if its scalar curvature is at least $n(n-1)$ with appropriate boundary conditions posed, where equality holds if and only the manifold is isometric to the canonical upper hemisphere. As an application, we derive an estimate for the area of event horizon in a vacuum static space with positive cosmological constant, which reveals an interesting connection between the area of event horizon and Brown-York mass. This estimate generalizes a similar result of Shen for three dimensional vacuum static spaces and also improves the uniqueness result of de Sitter space-time due to Hizagi-Montiel-Raulot. 
\end{abstract}

\maketitle



\section{Introduction}

In Riemannian geometry, one of the most remarkable result is the \emph{Lichnerowicz-Obata Theorem}, which says that the first eigenvalue of Laplacian on a closed manifold $(M^n,g)$ satisfies 
$$\lambda_1(M, g) \geq n,$$
if its Ricci curvature is assumed to be
$$Ric_g \geq (n-1) g.$$
Moreover, equality holds if and only if $(M,g)$ is isometric to the canonical sphere $\mathbb{S}^n$.\\

Besides this, it is well-known that Ricci curvature can also control other geometric invariants like volume, diameter \emph{etc}. These results have many important applications throughout the study of geometric analysis.\\ 

An interesting question is, can we derive similar results by posing assumptions on scalar curvature instead of Ricci curvature? In general, one might provide a negative answer since scalar curvature is conceptually a weaker notion compared with Ricci curvature. \\

However, extensive researches on scalar curvature shows that the geometry of manifolds with respect to positive scalar curvature is more rigid than we thought. Especially, the celebrated \emph{Positive Mass Theorem} proved by Schoen-Yau and Witten made one take a reconsideration on the role of scalar curvature in the study of differential geometry (cf. \cite{S-Y_1, S-Y_2, Witten}).\\

The main purpose of this article is to reveal the connection between first eigenfunctions and scalar curvature via positive mass theorem. In another word, we try to explore applications of positive mass theorem on eigenvalue problems. In fact, when studying the uniqueness of special vacuum static spaces, people have already noticed that first eigenfunctions have a deep connection with positive mass theorem. These works originated from Bunting and Masood-ul-Alam when they studied the uniqueness of Schwarzschild space-time (see \cite{B-M}). In their article, they found that the lapse function, which is a special harmonic function, can be used as a bridge between the original vacuum static space and the Euclidean space through a special conformal transformation. Similar ideas was also used in the work of Hizagi-Montiel-Raulot (\cite{H-M-R}) and Qing (\cite{Qing}) when studying the uniqueness of de Sitter/anti de Sitter space-time respectively. Notice that all these special lapse functions are in fact first eigenfunctions, we have reasons to believe that this should be a universal phenomenon instead of something rare. \\ 

In order to achieve what we expected, recall the \emph{positive mass theorem for Brown-York mass} first. This is the key in proving our central result.
\begin{theorem}[Shi and Tam \cite{S-T}]\label{thm:Shi-Tam}
For $n\geq 3$, suppose $(M^n, g)$ is a compact manifold with boundary $\Sigma:= \bigcup_{i=1}^m \Sigma_i$, where each $(\Sigma_i, g|_{\Sigma_i})$ is a connected component which can be embedded in $\mathbb{R}^n$ as a convex hypersurface.
Assume $3 \leq n \leq 7$ or $M$ is spin, moreover its scalar curvature
$$R_g \geq 0$$
and mean curvature of $\Sigma_i$ with respect to $g$ satisfies
$$H_g^i > 0$$ on $\Sigma_i$,
then the Brown-York mass
\begin{align}
m_{BY}(\Sigma_i, g) := \int_{\Sigma_i} \left( H_0^i - H_g^i \right) d\sigma \geq 0, \ \ \ i=1, \cdots, m
\end{align}
where $H_0^i$ is the mean curvature of $\Sigma_i$ with respect to the Euclidean metric. Moreover, if one of the mass vanishes then $(M,g)$ is isometric to a bounded domain in $\mathbb{R}^n$.
\end{theorem}

\begin{remark}
As claimed in \cite{Lohkamp_1, Lohkamp_2, S-Y_3} by Lohkamp and Schoen-Yau independently, for $n \geq 8$, the positive mass theorem for ADM mass is still valid. This would imply that assumptions on dimensions and spin structure in Theorem \ref{thm:Shi-Tam} can be removed, since it is equivalent to the positive mass of ADM mass. Thus, whenever applying Theorem \ref{thm:Shi-Tam} in this article, we always mean this enhanced version without explicitly mentioning it.
\end{remark}

\begin{remark}
	In fact, the isometrical embedding assumptions can be replaced by assumptions on secional curvatures, for instance Gaussian curvature of $\Sigma$ is positive when $n=3$. For the purpose of maintaining the simplicity of assamptions, we will still use this isometrical embedding assumption for our results throughout the whole article to avoid involving more complicated assumptions.  
\end{remark}

In exploring the relation between first eigenfunctions and scalar curvature, we discovered a set of constants defined for boundary components plays a crucial role in understanding it.

\begin{definition}
Let $\phi$ be the first eigenfunction
$$
\left\{  \aligned \Delta_g \phi + \lambda_1 \phi = 0, & \textit{\ \ \ on $M$}\\ 
\phi=0, & \textit{\ \ \ on $\Sigma$} \\ 
\endaligned\right.
$$
which is normalized such that
\begin{align}
\max_M \left( \frac{\lambda_1}{n}\phi^2 + |\nabla \phi|^2 \right) = 1.
\end{align}
We define constants
\begin{align}
\eta_g^i : = \min_{\Sigma_i}|\nabla \phi| \leq 1,
\end{align}
where $\Sigma_i$ is a connected component of $\partial M$.
\end{definition}

\begin{remark}
Since the space of first eigenfunctions is of dimension one, constants $\eta_g^i$ defined above in fact only depend on the manifold itself. In particular, we can see $\eta_g = 1$ for the canonical upper hemisphere $\mathbb{S}^n_+$, since $\lambda_1 = n$ and the quantity $\phi^2 + |\nabla \phi|^2 = 1$ constantly. 
\end{remark}

When trying to apply Theorem \ref{thm:Shi-Tam}, one always needs to assure that those geometric assumptions for boundary components are satisfied. For simplicity, we summarize them into the following regularity notion about hypersurfaces in manifold. Here we adapt the mean curvature condition for our purpose, since we won't apply Theorem \ref{thm:Shi-Tam} directly.
\begin{definition}
Let $\Sigma$ be a connected hypersurface in $(M, g)$. We say $\Sigma$ is \textbf{$\eta_g$-regular}, if it satisfies the following two assumptions:

\textbf{(A1)}: $(\Sigma, g|_{\Sigma})$ can be embedded in $\mathbb{R}^n$ as a convex hypersurface; 

\textbf{(A2)}: the mean curvature with respect to $g$ satisfies
\begin{align*}
H_g > - (n-1) \eta_g
\end{align*}
on $\Sigma$.
\end{definition}
Note that, (A2) will be satisfied automatically, if the mean curvature $H_g$ is non-negative, since $\eta_g$ is a constant between $0$ and $1$.\\

For a $\eta_g$-regular hypersurface $\Sigma$, we can formally define its \emph{Brown-York mass} to be
\begin{align}
m_{BY}(\Sigma, g):= \int_\Sigma (H_0 - H_g) d\sigma_g,
\end{align}
where $H_0$ is the mean curvature of $\Sigma$ when embedded in $\mathbb{R}^n$.\\

Our main result is the following inequality, which provides a characterization of interplays among scalar curvature, eigenfunctions and Brown-York mass:
\newtheorem*{thm_A}{\bf Theorem A}
\begin{thm_A}\label{thm:Dirichlet_eigenvalue_est}
For $n\geq 3$, let $(M^n, g)$ be an $n$-dimensional compact Riemannian manifold with boundary $\partial M = \bigcup_{i=1}^m \Sigma_i$, where each $\Sigma_i$ is a $\eta_g^i$-regular component. Suppose its scalar curvature satisfies 
\begin{align}
R_g \geq n(n-1)
\end{align}
and the first Dirichlet eigenvalue
$$\lambda_1(M,g) \leq n,$$ then we have inequalities
\begin{align}
||\partial_\nu \varphi||_{L^1(\Sigma_i, g)} \leq \frac{m_{BY} (\Sigma_i, g) }{n-1} \left( \max_M \left( \frac{\lambda_1}{n}\varphi^2 + |\nabla \varphi|^2\right)\right)^{\frac{1}{2}}, \ \ \ i = 1, \cdots, m
\end{align}
where $\nu  = - \frac{\nabla \varphi}{|\nabla \varphi|}$ is the outward normal of $\Sigma$ with respect to $g$. Moreover, equality holds for some $i_0$ if and only if $(M, g)$ is isometric to the canonical upper hemisphere $\mathbb{S}^n_+$.
\end{thm_A}

Immediately, we get an estimate for the first Dirichlet eigenvalue with respect to comparison of scalar curvature:
\newtheorem*{cor_A}{\bf Corollary A}
\begin{cor_A}\label{thm:Dirichlet_eigenvalue_est}
For $n\geq 3$, let $(M^n, g)$ be an $n$-dimensional compact Riemannian manifold with boundary $\partial M = \bigcup_{i=1}^m \Sigma_i$, where each $\Sigma_i$ is a $\eta_g^i$-regular component. Suppose its scalar curvature satisfies
\begin{align}
R_g \geq n(n-1)
\end{align}
 and there exists a component $\Sigma_{i_0}$ satisfies
\begin{align}
m_{BY}(\Sigma_{i_0}, g)  \leq (n-1) \eta^{i_0}_g |\Sigma_{i_0}|,
\end{align}
where $|\Sigma_{i_0}|$ is the $(n-1)$-dimensional measure of $\Sigma_{i_0}$.
Then we have the estimate for the first Dirichlet eigenvalue:
\begin{align}
\lambda_1 (M, g)\geq n,
\end{align}
where equality holds if and only if $(M, g)$ is isometric to the canonical upper hemisphere $\mathbb{S}^n_+$.
\end{cor_A}

As a special case, we have
\newtheorem*{cor_B}{\bf Corollary B}
\begin{cor_B}\label{thm:Dirichlet_eigenvalue_est}
For $n\geq 3$, let $(M^n, g)$ be an $n$-dimensional Riemannian manifold with boundary $\Sigma$. Suppose its scalar curvature satisfies
\begin{align}
R_g \geq n(n-1)
\end{align}
and its boundary $\Sigma$ is isometric to canonical sphere $\mathbb{S}^{n-1}$ and mean curvature satisfies that
\begin{align}
H_g \geq  (n-1) (1 -\eta_g).
\end{align}
Then we have the estimate for the first Dirichlet eigenvalue:
\begin{align}
\lambda_1 (M, g)\geq n,
\end{align}
where equality holds if and only if $(M, g)$ is isometric to the canonical upper hemisphere $\mathbb{S}^n_+$.
\end{cor_B}

\begin{remark}
By assuming $Ric_g \geq (n-1)g$, Reilly showed that $\lambda_1 \geq n$, if the boundary has nonnegative mean curvature (see \cite{Reilly}). Compared it with Corollary B, we require stronger assumptions on the boundary to compensate the weakness of scalar curvature.
\end{remark}

For a slightly more general result, please see Corollary \ref{cor:Dirichlet_eigenvalue_est_integral_mean_curvature}.\\

For another application, we show that Theorem A can help us to get an estimate for the event horizon in a vacuum static space with positive cosmological constant. \\

A \emph{vacuum static space} is a triple $(M^n, g, u)$, where $(M, g)$ is a Riemannian manifold and $u$ is a smooth nonnegative function which solves the \emph{vacuum static equation}
\begin{align}
\nabla^2 u - g \Delta_g u - u Ric_g = 0.
\end{align}
The set $\Sigma:= \{x \in M: u(x) = 0 \}$ is called the \emph{event horizon}, which can be easily shown to be a totally geodesic hypersurface. The reason we call it "vacuum static" is because the corresponding Lorentzian metric $\hat g = - u^2 dt^2 + g$ is a static solution to the \emph{vacuum Einstein equation} with cosmological constant $\Lambda$:
\begin{align}
Ric_{\hat g} - \frac{1}{2} R_{\hat g} \hat g + \Lambda \hat g = 0.
\end{align}
Note that, in this case the cosmological constant is given by $\Lambda = \frac{R_g}{2}$, thus in particular $\Lambda$ has the same sign with the scalar curvature of $(M, g)$. There are many interesting geometric features associates to vacuum static spaces, for more informations about vacuum static spaces, please refer to \cite{Q-Y_1, Q-Y_2}.\\

For three dimensional vacuum static spaces, Shen showed in \cite{Shen} that at least there is one of the components of the event horizon has area at most $4 \pi$, provided we normalize the scalar curvature to be $R_g = 6$. This result was improved by Ambrozio recently (see \cite{Ambrozio}). The idea of Shen's proof is based on an integral identity, which relates some positive terms inside the manifold with Euler characteristic on the boundary through Gauss-Bonnet formula. If one intends to estimate the area of each boundary component, suppose there are more than one, this way seems not working.\\

However, with the aid of Theorem A, it is possible for us to estimate the area of each component of the event horizon, provided these components are not too bad. This gives a different interpretation of the relation between the area of the event horizons and vacuum static space-times from the viewpoint of Brown-York mass.

\newtheorem*{thm_B}{\bf Theorem B}
\begin{thm_B}\label{thm:horizon_area_est}
For $n\geq 3$, let $(M^n, g, u)$ be an $n$-dimensional vacuum static space with scalar curvature $R_g = n(n-1)$ and event horizon $\Sigma = \bigcup_{i=1}^m \Sigma_i$, where each $\Sigma_i$ is a connected component. 

Suppose each $(\Sigma_i, g|_{\Sigma_i})$ can be isometrically embedded in $\mathbb{R}^n$ as a convex hypersurface, then we have
\begin{align}\label{ineq:horizon_area_est}
Area(\Sigma_i, g) \leq \frac{\max_i \kappa_i^2}{(n-1)(n-2) \kappa_i^2}\int_{\Sigma_i} \left( R_{\Sigma_i} + |\overset{\circ}{\overline A_i}|^2 \right)  d\sigma_g,
\end{align}
where $\kappa_i:= |\nabla u|_{\Sigma_i}$ is called the \emph{surface gravity of $\Sigma_i$}, $R_{\Sigma_i}$ is the intrinsic scalar curvature and $\overset{\circ}{\overline A_i}$ is the traceless second fundamental form of $\Sigma_i$ when embedded in $\mathbb{R}^n$. Moreover, equality holds for some $i_0$ if and only if $(M, g)$ is isometric to the canonical upper hemisphere $\mathbb{S}^n_+$, \emph{i.e.} a spatial slice of de Sitter space-time.
\end{thm_B}

\begin{remark}
	In general, the event horizon $\Sigma$ is not connected. There are many examples shows that $\Sigma$ has two connected components. The simplest one is the Nariai space $[0, \sqrt{3}\pi] \times \mathbb{S}^2(\frac{1}{\sqrt{3}})$. For more details about these examples, please see \cite{Q-Y_1}.
\end{remark}

\begin{remark}
Suppose the event horizon is connected and isometric to the canonical sphere $\mathbb{S}^{n-1}$, then equality holds in the inequality (\ref{ineq:horizon_area_est}) and hence rigidity holds. This recovers the uniqueness of de Sitter space-time in vacuum static space-time by Hizagi-Montiel-Raulot immediately (see \cite{H-M-R}).
\end{remark}

\begin{remark}
In fact, Shen's argument in \cite{Shen} can be easily generalized to higher dimensions. From this, under the assumption that $\Sigma$ is connected, the estimate (\ref{ineq:horizon_area_est}) can be improved to be
\begin{align}
Area(\Sigma_i, g) \leq \frac{1}{(n-1)(n-2)}\int_{\Sigma_i} R_{\Sigma_i}  d\sigma_g.
\end{align}
In particular, the right hand side equals $4\pi$ when $n=3$ due to Gauss-Bonnet formula and this is Shen's original result. 
\end{remark}

\paragraph{\textbf{Acknowledgement}}
The author would like to express his appreciations to Professor Chen Bing-Long, Dr. Fang Yi, Professor Huang Xian-Tao, Professor Qing Jie and Professor Zhang Hui-Chun for their inspiring discussions and comments. 


\section{A sharp inequality for first eigenfunctions}

Suppose $\varphi \in C^\infty(M)$ is an eigenfunction associated to the first Dirichlet eigenvalue $\lambda_1 > 0$. That is, $\varphi$ satisfies
\begin{align*}
\left\{ \aligned  \Delta_g \varphi + \lambda_1 \varphi &= 0, &\textit{on $M$}\\ 
 \varphi &= 0, &\textit{on $\Sigma$} .
\endaligned\right.
\end{align*}

Let
$$u:= \left(1 +  \alpha \varphi\right)^{-\frac{n-2}{2}} > 0,$$
where
$$\alpha :=\left( \max_M \left( \frac{\lambda_1}{n}\varphi^2 + |\nabla \varphi|^2\right)\right)^{-\frac{1}{2}}.$$ We consider a conformal metric $\hat g := u^{\frac{4}{n-2}} g$ with respect to the metric $g$. It has these following properties:

\begin{lemma}\label{lem:conformal_scalar_curvature}
Suppose the scalar curvature
$$R_g \geq n(n-1) $$
and the first eigenvalue
$$\lambda_1(M, g) \leq n,$$ 
then the scalar curvature of the conformal metric $\hat g$ satisfies that
$$R_{\hat g} \geq 0.$$
\end{lemma}

\begin{proof}
Denote $\psi:= \alpha \varphi$. From the well-known conformal transformation law of scalar curvature, we have
$$R_{\hat g} = u^{- \frac{n+2}{n-2}} \left( R_g u - \frac{4(n-1)}{(n-2)} \Delta_g u\right).$$
Since
\begin{align*}
\Delta_g (1 +   \psi)^{- \frac{n-2}{2}} =& - \frac{n-2}{2}   (1 +   \psi)^{- \frac{n}{2}} \Delta_g \psi + \frac{n(n-2)}{4}    (1 +   \psi)^{- \frac{n+2}{2}} |\nabla \psi|^2\\
=& \frac{n-2}{2}  \lambda_1 (1 +   \psi)^{- \frac{n}{2}} \psi + \frac{n(n-2)}{4}   (1 +   \psi)^{- \frac{n+2}{2}} |\nabla \psi|^2\\
=&\frac{n(n-2)}{4}   (1 +   \psi)^{- \frac{n+2}{2}} \left( \frac{2}{n} \lambda_1 (1 +   \psi) \psi +  |\nabla \psi|^2 \right),
\end{align*}
then
\begin{align*}
R_{\hat g} = &  (1 +   \psi)^{ \frac{n+2}{2}}\left( R_g \left(1 +   \psi\right)^{-\frac{n-2}{2}} - n(n-1)   (1 +  \psi)^{- \frac{n+2}{2}} \left( \frac{2}{n}\lambda_1 (1 +  \psi) \psi +  |\nabla \psi|^2 \right)\right)\\
=&  R_g \left(1 +   \psi\right)^2 - n(n-1)   \left( \frac{2}{n}\lambda_1 (1 +  \psi) \psi +  |\nabla \psi|^2 \right)\\
\geq& n(n-1) \left(\left(1 +   \psi\right)^2 -  \frac{2}{n}\lambda_1   (1 +   \psi) \psi -  |\nabla \psi|^2 \right).
\end{align*}

Note that
\begin{align*}
&\left(1 +   \psi\right)^2 -  \frac{2}{n}\lambda_1   (1 +   \psi) \psi \\
=& 1 - \frac{\lambda_1}{n} \left(  2(1 +   \psi) \psi - \frac{n}{\lambda_1} \left( \psi^2 + 2 \psi\right)\right) \\
\geq&  1 - \frac{\lambda_1}{n} \left(  2(1 +   \psi) \psi - \left( \psi^2 + 2 \psi\right)\right) \\
=& 1- \frac{\lambda_1}{n} \psi^2,
\end{align*}
since we assume that $\lambda_1 \leq n$.

Therefore,
\begin{align*}
R_{\hat g} \geq& n(n-1) \left( 1 - \left( \frac{\lambda_1}{n} \psi^2 + |\nabla \psi|^2\right)\right)= n(n-1) \left( 1 - \alpha^2\left( \frac{\lambda_1}{n} \varphi^2 + |\nabla \varphi|^2\right)\right) \geq 0.
\end{align*}
\end{proof}

As for the boundary, we have
\begin{lemma}\label{lem:conformal_boundary}
Let $\Sigma_i$ be a component of the boundary $\Sigma := \bigcup_{i=1}^m \Sigma_i$. Suppose $(\Sigma_i, g|_{\Sigma_i})$ is $\eta_g^i$-regular, then the mean curvature of $\Sigma_i$ with respect to the conformal metric $\hat g$ is positive and the Brown-York mass $$m_{BY} (\Sigma_i, \hat g) = m_{BY} (\Sigma_i, g) - (n-1) \alpha ||\partial_\nu \varphi||_{L^1(\Sigma_i, g)},$$
where $\nu = - \frac{\nabla \varphi}{|\nabla \varphi|}$ is the outward normal of $\Sigma_i$ with respect to the metric $g$.
\end{lemma}

\begin{proof}
When restricted on $\Sigma_i$,
$$u = 1,$$
since $\varphi = 0$ on $\Sigma_i$. Hence $(\Sigma, \hat g|_{\Sigma_i})$ is isometric to $(\Sigma_i, g|_{\Sigma_i})$, which can also be embedded in $\mathbb{R}^n$ as a convex surface with the same mean curvature, say $H_0^i$.

On the other hand, the mean curvature of $\Sigma_i$ with respect to $\hat g$ is given by 
\begin{align*}\label{ineq:mean_curvature}
H_{\hat g}^i = u^{- \frac{2}{n-2}} \left( H_g^i + \frac{2(n-1)}{n-2} \partial_\nu \log u \right) = H_g^i + (n-1)  \alpha |\nabla \varphi|_{\Sigma_i} \geq  H_g^i + (n-1)   \eta_g^i > 0,
\end{align*}
where $\nu$ is the outward normal of the boundary $\Sigma_i$ and we used the fact that
$$\alpha |\nabla \varphi|_{\Sigma_i} \geq \eta_g^i .$$
Thus, the Brown-York mass of $(\Sigma_i, \hat g)$ is given by
\begin{align*}
m_{BY}(\Sigma_i, \hat g) =& \int_{\Sigma_i} \left( H_0^i - H_{\hat g}^i \right) d\sigma_g\\
=& \int_{\Sigma_i} \left( H_0^i - H_g^i \right) d\sigma_g - (n-1) \alpha \int_{\Sigma_i} |\nabla \varphi| d\sigma_g \\
=& m_{BY} (\Sigma_i, g) - (n-1) \alpha ||\partial_\nu \varphi||_{L^1(\Sigma_i, g)}.
\end{align*}
\end{proof}

Now we give the proof of our main theorem.

\begin{proof}[Proof of Theorem A]
By Lemma \ref{lem:conformal_scalar_curvature}, we know that the conformal metric $\hat g$ has non-negative scalar curvature. Form $\eta_g^i$-regularity assumptions and Lemma \ref{lem:conformal_boundary}, all components can be isometrically embedded in $\mathbb{R}^n$ as convex hypersurfaces and the mean curvature is positive with respect to $\hat g$. Applying Theorem \ref{thm:Shi-Tam}, we conclude that
$$m_{BY} (\Sigma_i, \hat g) = m_{BY} (\Sigma_i, g) - (n-1) \alpha ||\partial_\nu \varphi||_{L^1(\Sigma_i, g)} \geq 0, \ \ \ i = 1, \cdots, m$$ That is,
$$||\partial_\nu \varphi||_{L^1(\Sigma_i, g)} \leq \frac{m_{BY} (\Sigma_i, g) }{(n-1) \alpha} = \frac{m_{BY} (\Sigma_i, g) }{n-1} \left( \max_M \left( \frac{\lambda_1}{n}\varphi^2 + |\nabla \varphi|^2\right)\right)^{\frac{1}{2}}, \ \ \ i= 1, \cdots, m.$$
 
If $(M,g)$ is isometric to the canonical upper hemisphere, its Brown-York mass is given by 
$$m_{BY}(\Sigma, g) = \int_{\mathbb{S}^{n-1}} (n-1) d\sigma_{g_{\mathbb{S}^{n-1}}} = (n-1) \omega_{n-1},$$
where $\omega_{n-1}$ is the volume of round sphere $\mathbb{S}^{n-1}$. Meanwhile,
$$\frac{\lambda_1}{n} \varphi^2 + |\nabla \varphi|^2$$ 
is a constant on $M$ and in particular,
$$|\nabla \varphi|_\Sigma = \left.\left(\frac{\lambda_1}{n} \varphi^2 + |\nabla \varphi|^2 \right) \right|_\Sigma$$ is a constant on $\Sigma$. Therefore,
$$\frac{m_{BY} (\Sigma, g) }{n-1} \left( \max_M \left( \frac{\lambda_1}{n}\varphi^2 + |\nabla \varphi|^2\right)\right)^{\frac{1}{2}} = \omega_{n-1}|\nabla \varphi|_\Sigma =  ||\partial_\nu \varphi||_{L^1(\Sigma, g)}.$$
\emph{i.e.} equality holds in this case.

Conversely, if there exists an $i_0$ such that 
$$ ||\partial_\nu \varphi||_{L^1(\Sigma_{i_0}, g)} = \frac{m_{BY} (\Sigma_{i_0}, g) }{n-1} \left( \max_M \left( \frac{\lambda_1}{n}\varphi^2 + |\nabla \varphi|^2\right)\right)^{\frac{1}{2}}$$
holds, then
$$m_{BY}(\Sigma_{i_0}, \hat g) = 0$$ 
and hence $(M, \hat g)$ is isometric to a connected compact domain of $\mathbb{R}^n$ by the rigidity of Brown-York mass. In particular, the scalar curvature $R_{\hat g} = 0$, which implies that $R_g = n(n-1)$, $\lambda_1(M,g) = n$ and $$\varphi^2 + |\nabla \varphi|^2 = \alpha^{-2}$$ is a constant on $M$ by checking the proof of Lemma \ref{lem:conformal_scalar_curvature}. 

Now we show that $(M, g)$ has to be isometric to the standard upper hemisphere $\mathbb{S}^n_+$, where the essential idea is an adapted version of the proof of \emph{obata theorem} (cf. Proposition 3.1 in \cite{L-P}).

For simplicity, we denote 
$$w:= 1 + \alpha\varphi \geq 1$$
and hence $\hat g = w^{-2} g$. From the conformal transformation law of Ricci tensor and scalar curvature, the traceless Ricci tensor is given by
$$E_{\hat g} = E_g + (n-2) w^{-1} \left( \nabla^2 w - \frac{1}{n} g \Delta_g w \right) = 0.$$
That is,
$$E_g = - (n-2) w^{-1} \left( \nabla^2 w - \frac{1}{n} g \Delta_g w \right).$$
Thus, we have
\begin{align*}
\int_M w |E_g|^2 dv_g =& - (n-2) \int_M \langle \nabla^2 w - \frac{1}{n} g \Delta_g w , E_g \rangle dv_g \\
=& (n-2) \int_M \langle \nabla w, div_g E_g \rangle dv_g - (n-2) \int_\Sigma E_g (\nabla w, \nu) d\sigma_g \\
=& - (n-2) \int_\Sigma E_g (\nabla w, \nu) d\sigma_g,
\end{align*}
where $$div_g E_g = \frac{n-2}{2n} dR_g = 0$$
by the \emph{contracted second Bianchi identity}.

On the other hand, 
\begin{align*}
E_g (\nabla w, \nu) =& - (n-2) w^{-1} \left( \nabla^2 w (\nabla w, \nu)  - \frac{1}{n} g(\nabla w, \nu)  \Delta_g w \right) \\
=& - (n-2)  \alpha^2\left( - |\nabla \varphi|_g^{-1} \nabla^2 \varphi (\nabla \varphi , \nabla \varphi)  + \frac{1}{n} |\nabla \varphi|_g  \Delta_g \varphi \right) \\
=&  (n-2) \alpha^2 \left( |\nabla \varphi|_g^{-1} \nabla^2 \varphi (\nabla \varphi , \nabla \varphi)  + \varphi |\nabla \varphi|_g \right) \\
=& - \frac{n-2}{2} \alpha^2 \nabla_\nu |\nabla \varphi|^2
\end{align*}
on $\Sigma$. Since $$\varphi^2 + |\nabla \varphi|^2  = \alpha^{-2}$$ on $M$, we have $$\frac{d}{dt} \left(\varphi^2 + |\nabla \varphi|^2 \right) (\gamma(t)) = 0,$$ where $\gamma: [0, 1) \rightarrow M$ is a geodesic starting from an arbitrary point $p = \gamma(0) \in \Sigma$ with $\gamma'(0) = - \nu(p)$. Thus,
$$ \nabla_\nu |\nabla \varphi|^2 (p) = -\left.\frac{d}{dt}\right|_{t = 0} |\nabla \varphi|^2  (\gamma(t))  =   \left.\frac{d}{dt}\right|_{t = 0}  \varphi (\gamma(t))^2 = 2 \varphi (p) |\nabla \varphi |(p) = 0 $$ since $\varphi(p) = 0$.

Therefore, 
$$\int_M w |E_g|^2 dv_g = \frac{(n-2)^2}{2} \alpha^2 \int_\Sigma \nabla_\nu |\nabla \varphi|^2 d\sigma_g =0$$
and hence $(M, g)$ is Einstein. Together with the fact that $(M, g)$ is conformally flat and scalar curvature $R_g = n(n-1)$, we conclude that $(M, g)$ is isometric to a connected compact domain in the round sphere $\mathbb{S}^n$. 

On the other hand, since the boundary $\Sigma$ can be embedded in $\mathbb{R}^n$ as a hypersurface, $\Sigma$ can be obtained by intersecting $\mathbb{S}^n$ with a hyperplane in $\mathbb{R}^{n+1}$, which shows that $\Sigma$ is isometric to an $(n-1)$-sphere $\mathbb{S}^{n-1}(r)$ with radius $r \leq 1$ and hence $(M,g)$ is isometric to a geodesic ball in $\mathbb{S}^n$. Now $(M,g)$ is isometric to the canonical upper hemisphere $\mathbb{S}^n_+$, since the first eigenvalue $\lambda_1(M,g) = 0$.
\end{proof}


\section{Applications of the main theorem}

If we assume comparison of Ricci curvature in Theorem A instead of scalar curvature, we would get a sharp inequality for the normal derivative for first eigenfunctions:

\begin{corollary}\label{cor:Dirichlet_eigenvalue_est_Ricci}
For $n\geq 3$, let $(M^n, g)$ be an $n$-dimensional compact Riemannian manifold with connected $\eta_g$-regular boundary $\Sigma$. Suppose its Ricci curvature satisfies that
\begin{align}
Ric_g \geq (n-1)g
\end{align}
and the first Dirichlet eigenvalue
$$\lambda_1(M,g) \leq n,$$ then we have inequality
\begin{align}
||\partial_\nu \varphi||_{L^1(\Sigma, g)} \leq \frac{m_{BY} (\Sigma, g) }{(n-1)} ||\partial_\nu \varphi||_{L^\infty(\Sigma, g)},
\end{align}
where equality holds if and only if $(M, g)$ is isometric to the canonical upper hemisphere $\mathbb{S}^n_+$.
\end{corollary}

\begin{remark}
It is well-known that if $Ric_g \geq (n-1)g$ and $H_g \geq 0$ on $\Sigma$, then the eigenvalue estimate $$\lambda_1(M,g) \geq n$$ holds (see \cite{Reilly}). Clearly, this estimate won't be true, if $H_g < 0$ on $\Sigma$ by checking geodesic balls in $\mathbb{S}^n$ with radius strictly greater than $\frac{\pi}{2}$. 
\end{remark}

Corollary \ref{cor:Dirichlet_eigenvalue_est_Ricci} can be deduced easily from Theorem A together with the following well-known fact:
\begin{lemma}\label{lem:Ric_subharmonic}
Suppose $Ric_g \geq (n-1) g$ and $\lambda_1 \leq n$, then $W:=\frac{\lambda_1}{n} \varphi^2 +  |\nabla \varphi|^2$ is a subharmonic function.
\end{lemma}

\begin{proof}
From \emph{Bochner's formula},
\begin{align*}
\frac{1}{2} \Delta_g W &= \frac{\lambda_1}{n} \varphi \Delta_g \varphi +\frac{\lambda_1}{n}  |\nabla \varphi|^2 + \frac{1}{2} \Delta_g |\nabla \varphi|^2 \\
&= - \frac{\lambda_1^2}{n} \varphi^2 + \frac{\lambda_1}{n} |\nabla \varphi|^2 + |\nabla^2 \varphi|^2 + \langle \nabla \varphi, \nabla \Delta_g \varphi \rangle + Ric_g (\nabla \varphi, \nabla \varphi) \\
&= - \frac{\lambda_1^2}{n} \varphi^2 -  \frac{(n - 1)\lambda_1}{n}|\nabla \varphi|^2 + \left|\nabla^2 \varphi - \frac{1}{n} g \Delta_g \varphi \right|^2  + \frac{1}{n} (\Delta_g \varphi)^2 + Ric_g (\nabla \varphi, \nabla \varphi) \\
&\geq - \frac{\lambda_1^2}{n} \varphi^2 -  \frac{(n - 1)\lambda_1}{n}|\nabla \varphi|^2  + \frac{1}{n} (\Delta_g \varphi)^2 + (n-1) |\nabla \varphi|^2\\
&=  (n - 1) \left( 1 - \frac{\lambda_1}{n}\right) |\nabla \varphi|^2 \\
& \geq 0,
\end{align*}
by the assumption that $\lambda_1 \leq n$.
\end{proof}

\begin{proof}[Proof of Corollary \ref{cor:Dirichlet_eigenvalue_est_Ricci}]
From Lemma \ref{lem:Ric_subharmonic}, $$W=\frac{\lambda_1}{n} \varphi^2 +  |\nabla \varphi|^2$$ is a subharmonic function on $M$. By \emph{maximum principle}, $$\max_M W = \max_\Sigma W = ||\partial_\nu \varphi||^2_{L^\infty(\Sigma, g)}.$$
Now applying Theorem A, we get
\begin{align*}
||\partial_\nu \varphi||_{L^1(\Sigma, g)} \leq \frac{m_{BY} (\Sigma, g) }{n-1} \left( \max_M  W\right)^{\frac{1}{2}} = \frac{m_{BY} (\Sigma, g) }{n-1} ||\partial_\nu \varphi||_{L^\infty(\Sigma, g)},
\end{align*}
where equality holds if and only if $(M,g)$ is isometric to the canonical upper hemisphere.
\end{proof}

By taking the contraposition of Theorem A, we are readily to deduce Corollary A:
\begin{proof}[Proof of Corollary A]
We normalize the eigenfunction $\varphi$ such that
$$\max_M \left(\frac{\lambda_1}{n} \varphi^2 +  |\nabla \varphi|^2 \right) = 1.$$
Then
$$||\partial_\nu \varphi||_{L^1(\Sigma_i, g)} \geq \left(\min_{\Sigma_i} |\nabla \varphi| \right) |\Sigma_i| = \eta^i_g |\Sigma_i|, i =1, \cdots, m.$$
Suppose $\lambda_1(M,g) < n$ and there exists an $i_0$ such that
 $$ m_{BY}(\Sigma_{i_0}, g) \leq (n-1)\eta^{i_0}_g |\Sigma_{i_0}|,$$
 then we have
 $$||\partial_\nu \varphi||_{L^1(\Sigma_{i_0}, g)} \geq \eta^{i_0}_g |\Sigma_{i_0}| \geq \frac{m_{BY}(\Sigma_{i_0}, g) }{n-1}.$$
 Now applying Theorem A, we deduce that $(M,g)$ is isometric to the canonical upper hemisphere $\mathbb{S}^n_+$ and hence $\lambda_1(M,g) = n$. This contradicts to the assumption that $\lambda_1(M,g) < n$ and hence we have $\lambda_1(M,g) \geq n$. The rigidity part follows from the corresponding one in Theorem A.
\end{proof}

As a very special case, Corollary A is readily to imply Corollary B. However, let us state and prove a slightly more general form instead.
\begin{corollary}\label{cor:Dirichlet_eigenvalue_est_integral_mean_curvature}
For $n\geq 3$, let $(M^n, g)$ be an $n$-dimensional Riemannian manifold with boundary $\Sigma := \partial M$. Suppose its scalar curvature satisfies 
\begin{align}
R_g \geq n(n-1)
\end{align}
and the boundary $\Sigma$ is isometric to the canonical $\mathbb{S}^{n-1}$ in $\mathbb{R}^n$. Moreover, assume its mean curvature satisfies 
\begin{align}
H_g > - (n-1) \eta_g
\end{align}
on $\Sigma$ and
\begin{align}
\int_{\Sigma} H_g d\sigma_g \geq (n-1) (1 - \eta_g) \omega_{n-1},
\end{align}
where $\omega_{n-1}$ is the volume of canonical sphere $\mathbb{S}^{n-1}$.
Then we have the eigenvalue estimate:
\begin{align}
\lambda_1 (M, g)\geq n,
\end{align}
where equality holds if and only if $(M, g)$ is isometric to the canonical upper hemisphere $\mathbb{S}^n_+$.
\end{corollary}

\begin{proof}
The conclusion simply follows from Theorem B together with the fact that 
$$m_{BY}(\Sigma, g) = \int_\Sigma (H_0 - H_g) d\sigma_g = (n-1) \omega_{n-1} - \int_\Sigma  H_g d\sigma_g,$$ where $H_0 = n-1$ is the mean curvature of canonical sphere $\mathbb{S}^{n-1}$ in $\mathbb{R}^n$.
\end{proof}


\section{An area estimate for the event horizon}

In the end, we give the estimate of the area of the event horizon in a vacuum static space with positive cosmological constant.
\begin{proof}[Proof of Theorem B]
Since $(M, g, u)$ is a vacuum static space, then by definition, the lapse function $u$ is non-negative and vanishes only on $\Sigma$. Moreover, it satisfies the \emph{vacuum static equation}
\begin{align}\label{eqn:vacuum_static}
\nabla^2 u - g \Delta_g u - u Ric_g = 0.
\end{align}
By taking trace of the above equation, we get
$$\Delta_g u + \frac{R_g}{n-1} u = \Delta_g u + n u =0.$$
From the positivity of $u$ on the interior of $M$, we deduce that $u$ is a first eigenfunction with eigenvalue $\lambda_1(M,g) = n$.

In \cite{Shen}, Shen showed that the following \emph{Robinson-type identity}
\begin{align}\label{eqn:div}
div_g \left( u^{-1} \nabla(u^2 + |\nabla u|^2 )\right) = 2u|E_g|^2
\end{align}
holds on $M \backslash \Sigma$. Then applying \emph{maximum principle}, we see
$$\max_M ( u^2 +|\nabla u|^2 ) = \max_\Sigma ( u^2 +|\nabla u|^2 ) = \max_\Sigma |\nabla u|^2 .$$

On the other hand, $\nabla^2 u$ vanishes on $\Sigma$, it implies that $|\nabla u|$ is a constant on $\Sigma$ and $\Sigma$ itself is totally geodesic which in particular has vanishing mean curvature $H_g$. Hence by Theorem A, we have
\begin{align*}
||\partial_\nu u||_{L^1(\Sigma_i, g)} \leq \frac{m_{BY} (\Sigma_i, g) }{(n-1)} \left( \max_M \left( u^2 + |\nabla u|^2\right)\right)^{\frac{1}{2}} = \frac{\max_\Sigma |\nabla u|}{n-1}  \int_{\Sigma_i} \overline H_i d\sigma_g 
\end{align*}
for $i = 1, \cdots, m$, where $\overline H_i$ is the mean curvature of $\Sigma_i$ induced by the Euclidean metric. That is,
$$Area(\Sigma_i, g) \leq  \frac{\max_i \kappa_i}{(n-1)\kappa_i} \int_{\Sigma_i} \overline H_i d\sigma_g, \ \ \ i = 1, \cdots, m.$$

From Gauss equation in $\mathbb{R}^n$, the scalar curvature of $\Sigma_i$ satisfies that
$$R_{\Sigma_i} = \overline H^2_i - |\overline{A}_i|^2 = \frac{n-2}{n-1}\overline H^2_i - |\overset{\circ}{\overline A_i}|^2,$$
where $\overline A_i$ and $\overset{\circ}{\overline A_i}$ are second fundamental form and its traceless part respectively. Thus by H\"older's inequality, we get
\begin{align*}
Area(\Sigma_i, g) \leq  \frac{\max_i \kappa_i^2}{(n-1)^2 \kappa_i^2}\int_{\Sigma_i} \overline H_i^2 d\sigma_g = \frac{\max_i \kappa_i^2}{(n-1)(n-2) \kappa_i^2}\int_{\Sigma_i} \left( R_{\Sigma_i} + |\overset{\circ}{\overline A_i}|^2 \right) d\sigma_g,
\end{align*}
for $i = 1, \cdots, m.$

The rigidity follows from the corresponding part of Theorem A.
\end{proof}


\bibliographystyle{amsplain}

\end{document}